\documentclass[11pt]{amsart}
\pdfoutput=1
\usepackage{amssymb,amscd,graphicx,enumerate,epsfig}
\usepackage[bookmarks=true]{hyperref}

\newcommand{\Claim}[1]{\noindent\textbf{Claim. }\textit{#1}\\}
\newtheorem{theorem}{Theorem}[section]

\newtheorem{corollary}[theorem]{Corollary}

\newcommand{\Corollaryno}[1]{\noindent\textbf{Corollary #1.}}

\theoremstyle{definition}
\newtheorem{definition}[theorem]{Definition}
\newenvironment{proofc}{\noindent\textit{Proof of Claim.}}{\\}

\begin{document}

\title{A note on the nearly additivity of knot width}
\author{Jungsoo Kim}
\date{13, Jan, 2010}
\begin{abstract}
Let $k$ be a knot in $S^3$. In \cite{HS}, H.N. Howards and J. Schultens introduced a method to construct a manifold decomposition of double branched cover of $(S^3, k)$ from a thin position of $k$. In this article,
we will prove that if a thin position of $k$ induces a thin decomposition of double branched cover of $(S^3,k)$ by Howards and Schultens' method, then the thin position is the sum of prime summands by stacking  a thin position of one of prime summands of $k$ on top of a thin position of another prime summand, and so on. Therefore, $k$ holds the nearly additivity of knot width  (i.e. for $k=k_1\#k_2$, $w(k)=w(k_1)\# w(k_2)-2$) in this case. Moreover, we will generalize the hypothesis to the property a thin position induces a manifold decomposition whose thick surfaces consists of strongly irreducible or critical surfaces (so topologically minimal.)
\end{abstract}

\address{\parbox{4in}{Department of Mathematics\\
Konkuk University\\
Seoul 143-701\\ Korea\medskip}} \email{pibonazi@gmail.com}

\maketitle
\tableofcontents

\section{Introduction and result}

Let $k$ be a knot in $S^3$. In \cite{HS}, H.N. Howards and J. Schultens introduced a method to construct a manifold decomposition of double branched cover (abbreviate it as \textit{DBC}, and call the method \textit{the H-S method}) of $(S^3, k)$ (see section \ref{section-DBC}) and they proved that for $2$-bridge knots and $3$-bridge knots in thin position DBC inherits thin manifold decomposition (note that a knot in a thin position may not induce a thin manifold decomposition by the H-S method in general, see \cite{HS} and \cite{HRS}.) Indeed, if $k$ is a non-prime $3$-bridge knot, then $k=k_1\#k_2$ for $2$-bridge knots $k_1$ and $k_2$, and thin position of $k$ is the sum of $k_1$ and $k_2$ by stacking a thin position of one of the knots on top of a thin position of the other (see Corollary 2.5 of \cite{HS}.) So $w(k)=w(k_1)+w(k_2)-2$, i.e. $k$ holds the \textit{nearly additivity} of knot width (for $k=k_1\#k_2$, $w(k)=w(k_1)+w(k_2)-2$, see \cite{ST2} for more details on ``the nearly additivity of knot width''.)

So we get a question whether the property that a thin position of a knot induces a thin manifold decomposition of DBC by the H-N method implies the nearly additivity of knot width in general.
If a thin position of $k$ is the sum of an ordered stack of prime summands of $k$ where each summand is in a thin position (``\textit{a sum of an ordered stack of prime summands}'' means like the left of Figure \ref{figure-thin}, where the bottom summand is a Montesinos knot $M(0; (2,1), (3,1), (3, 1), (5, 1))$ in a thin position (this figure is borrowed from Figure 5.2.(c) of \cite{HK}) and the top summands are trefoils in thin position,)
then the sum of an ordered stack like that in a different order also determines a thin position of $k$. So $k$ must hold the nearly additivity of knot width by the uniqueness of prime factorization of $k$.
In \cite{RS}, Y. Rieck and E. Sedgwick proved that thin position of the sum of small knots is the sum of an ordered stack in that manner, i.e. it holds the nearly additivity of knot width. But M. Scharlemann and A. Thompson proposed a way to construct a example to contradict the nearly additivity of knot width (see \cite{ST2}.) Although R. Blair and M. Tomova proved that most of Scharlemann and Thompson's constructions do not produce counterexamples for the nearly additivity of knot width (see \cite{BT},) the question seems not obvious.

In this article, we will prove that the question is true.

\begin{theorem}\label{theorem-1}
	If a thin position of a knot $k$ induces a thin manifold decomposition of double branched cover of $(S^3,k)$ by the Howards and Schultens' method, then the thin position is the sum of prime summands by stacking  a thin position of one of prime summands of $k$ on top of a thin position of another prime summand, and so on. Therefore, $k$ holds the nearly additivity of knot width in this case.
\end{theorem}

In section \ref{section-critical}, we will generalize Theorem 1.1 by using the concept \textit{critical surface} originated from D. Bachman (see \cite{Bachman1} and \cite{Bachman3} for the original definition and the recently modified definition of ``critical surface''.) So we will get the corollary.\\

\Corollaryno{5.3}\textit{
If a thin position of a knot $k$ induces a manifold decomposition of double branched cover $M$ of $(S^3,k)$ by the Howards and Schultens' method where each thick surface $H_+$ of the manifold decomposition of $M$ is strongly irreducible or critical in $M(H_+)$, then the thin position of $k$ is the sum of prime summands by stacking  a thin position of one of prime summands on top of a thin position of another prime summand, and so on. Therefore, $k$ holds the nearly additivity of knot width in this case.}

\section{Generalized Heegaard splittings}

In this section, we will introduce some definitions about generalized Heegaard splittings. We use the notations and definitions by D. Bachman in \cite{Bachman3} through this section for convenience.

\begin{definition} 
A \textit{compression body} (a \textit{punctured compression body} resp.) is a $3$-manifold which can be obtained by starting with some closed, orientable, connected surface, $H$, forming the product $H\times I$, attaching some number of $2$-handles to $H\times\{1\}$ and capping off all resulting $2$-sphere boundary components (some $2$-sphere boundaries resp.) that are not contained in $H\times\{0\}$ with $3$-balls. The boundary component $H\times\{0\}$ is referred to as $\partial_+$. The rest of the boundary is referred to as $\partial_-$. 
\end{definition}

\begin{definition} 
A \textit{Heegaard splitting} of a $3$-manifold $M$ is an expression of $M$ as a union $V\cup_H W$, where $V$ and $W$ are compression bodies that intersect in a transversally oriented surface $H=\partial_+V=\partial_+ W$. If $V\cup_H W$ is a Heegaard splitting of $M$ then we say $H$ is a \textit{Heegaard surface}.
\end{definition}

\begin{definition} 
Let $V\cup_H W$ be a Heegaard splitting of a $3$-manifold $M$. Then we say the pair $(V,W)$ is a \textit{weak reducing pair} for $H$ if $V$ and $W$ are disjoint compressing disks on opposite sides of $H$. A Heegaard surface is \textit{strongly irreducible} if it is compressible to both sides but has no weak reducing pairs.
\end{definition}

\begin{definition} \label{definition-GHS}
A \textit{generalized Heegaard splitting} (GHS)\footnote{Note that D. Bachman did not allow thin surfaces in a GHS to be $2$-spheres in \cite{Bachman3}. In particular, he introduced more generalized concept ``pseudo-GHS'' in \cite{Bachman3} to deal with thin spheres and trivial compression bodies.} H of a $3$-manifold $M$ is a pair of sets of pairwise disjoint, transversally oriented, connected surfaces, $\operatorname{Thick}(H)$ and $\operatorname{Thin}(H)$ (in this article, we will call the elements of each of both \textit{thick surfaces} and \textit{thin surfaces}, resp.), which satisfies the following conditions.
	\begin{enumerate}
		\item Each component $M'$ of $M-\operatorname{Thin}(H)$ meets a unique element $H_+$ of $\operatorname{Thick}(H)$ and $H_+$ is a Heegaard surface in $M'$. Henceforth we will denote the closure of the component of $M-\operatorname{Thin}(H)$ that contains an element $H_+\in\operatorname{Thick}(H)$ as $M(H_+)$.
		\item As each Heegaard surface $H_+\subset M(H_+)$ is transversally oriented, we can consistently talk about the points of $M(H_+)$ that are ``above'' $H_+$ or ``below'' $H_+$. Suppose $H_-\in \operatorname{Thin}(H)$. Let $M(H_+)$ and $M(H'_+)$ be the submanifolds on each side of $H_-$. Then $H_-$ is below $H_+$ if and only if it is above $H'_+$.
		\item There is a partial ordering on the elements of $\operatorname{Thin}(H)$ which satisfies the following: Suppose $H_+$ is an element of $\operatorname{Thick}(H)$, $H_-$ is a component of $\partial M(H_+)$ above $H_+$ and $H'_-$ is a component of $\partial M(H_+)$ below $H_+$. Then $H_->H'_-$.
	\end{enumerate}
\end{definition}

\begin{definition} 
Suppose $H$ is a GHS of a $3$-manifold $M$ with no $S^3$ components. Then $H$ is \textit{strongly irreducible} if each element $H_+\in \operatorname{Thick}(H)$ is strongly irreducible in $M(H_+)$.
\end{definition}

\section{The manifold decomposition of double branched cover of $(S^3,k)$ from a thin position of $k$\label{section-DBC}}

In this section, we will describe the method to construct a manifold decomposition of DBC from a thin position of a knot originated from H.N. Howards and J. Schultens in \cite{HS} (see section 3 of \cite{HS} for more details.) We borrow the notions and definitions directly from \cite{HS}.

\begin{definition} 
Let $h:\{S^3 - (\text{two points})\} \to [0, 1]$ be a height function on $S^3$ that restricts to a Morse function on $k$. Choose a regular value $t_i$ between each pair of adjacent critical values of $h|k$. The \textit{width} of $k$ with respect to $h$ is $\sum_i \#|k\cap h^{-1}(t_i)|$.
Define the \textit{width} of $k$ to be the minimum width of $k$ with respect to $h$ over all $h$. A \textit{thin position} of $k$ is the presentation of $k$ with respect to a height function that realizes the width of $k$.
\end{definition}

\begin{definition} 
A \textit{thin level} for $k$ is a $2$-sphere $S$ such that the following hold: 
	\begin{enumerate}
		\item $S = h^{-1}(t_0)$ for some regular value $t_0$;
		\item $t_0$ lies between adjacent critical values $x$ and $y$ of $h$, where $x$ is a minimum of $k$ lying above $t_0$ and $y$ is a
maximum of $k$ lying below $t_0$. 
	\end{enumerate}
A \textit{thick level} is a $2$-sphere $S$ such that the following hold: 
	\begin{enumerate}
		\item $S = h^{-1}(t_0)$ for some regular value $t_0$;
		\item $t_0$ lies between adjacent critical values $x$ and $y$ of $h$, where $x$ is a maximum of $k$ lying above $t_0$ and $y$ is a
minimum of $k$ lying below $t_0$.
	\end{enumerate}
\end{definition}

\begin{definition}
A \textit{manifold decomposition} is a generalized version of GHS in Definition \ref{definition-GHS}\footnote{Many authors use the term ``generalized Heegaard splitting'' to denote ``manifold decomposition'' in \cite{ST} by M. Scharlemann and A. Thompson, but the author distinguished the terms ``manifold decomposition'' and ``generalized Heegaard splitting'' to use two different definitions at the same time.}, where we permit thin surfaces to be 2-spheres (this ``manifold decomposition'' is originated from \cite{ST} by M. Scharlemann and A. Thompson.)
So adjacent thick and thin surfaces in a manifold decomposition cobound a (possibly, punctured-) compression body.
\end{definition}

\begin{definition} 
Let the \textit{complexity} of a connected surface $S$ be $c(S)= 1-\chi(S) = 2 \operatorname{genus}(S)-1$ for $S$ of positive genus. Define $c(S^2) = 0$. For $S$ not necessarily connected define $c(S) = \sum\{c(S')| S'\text{ a connected component of }S\}$.
\end{definition}

\begin{definition}
Let the \textit{width} of the manifold decomposition $H$ of $M$ be the set of integers $\{c(S_i)| 1\leq i \leq k, \text{ each }S_i\text{ is the thick surface of }H \}$.
\end{definition}

Order these integers in monotonically non-increasing order. Compare the ordered multi-sets lexicographically.

\begin{definition} 
Define the \textit{width} $w(M)$ of $M$ to be the minimal width over all manifold decompositions using the above ordering of the sets of integers.
\end{definition}

\begin{definition}
A given manifold decomposition of $M$ is \textit{thin} if the width of the manifold decomposition is the width of $M$.
\end{definition}

Let $k$ be a knot and denote DBC of $(S^3,k)$ by $M$.
If $k$ is in a thin position, then $M$ inherits a manifold decomposition as follows: Denote the thick levels of $k$ by $S_1,\cdots,S_n$ and the thin levels by $L_1,\cdots,L_{n-1}$. Each $S_i$ and each $L_i$ is a sphere that meets the knot some (even) number of times. More specifically, each $S_i$ meets $k$ at least $4$ times and each $L_i$ meets $k$ at least $2$ times. Denote the surface in $M$ corresponding to $S_i$ by $\tilde{S}_i$ and the surface in $M$ corresponding to $L_i$ by $\tilde{L}_i$ . Each $\tilde{S}_i$ is a closed orientable surface of genus at least $1$ and each $\tilde{L}_i$ is a closed orientable surface. More specifically, if $S_i$ meets $k$ exactly $2l$ times, then $\tilde{S}_i$ is a closed orientable surface of genus $l-1$. And if $L_i$ meets $k$ exactly $2l$ times, then $\tilde{L}_i$ is a closed orientable surface of genus $l-1$. See Figure \ref{figure-thin}.

\begin{figure}
\includegraphics[bb=48 170 503 572,width=8.5cm]{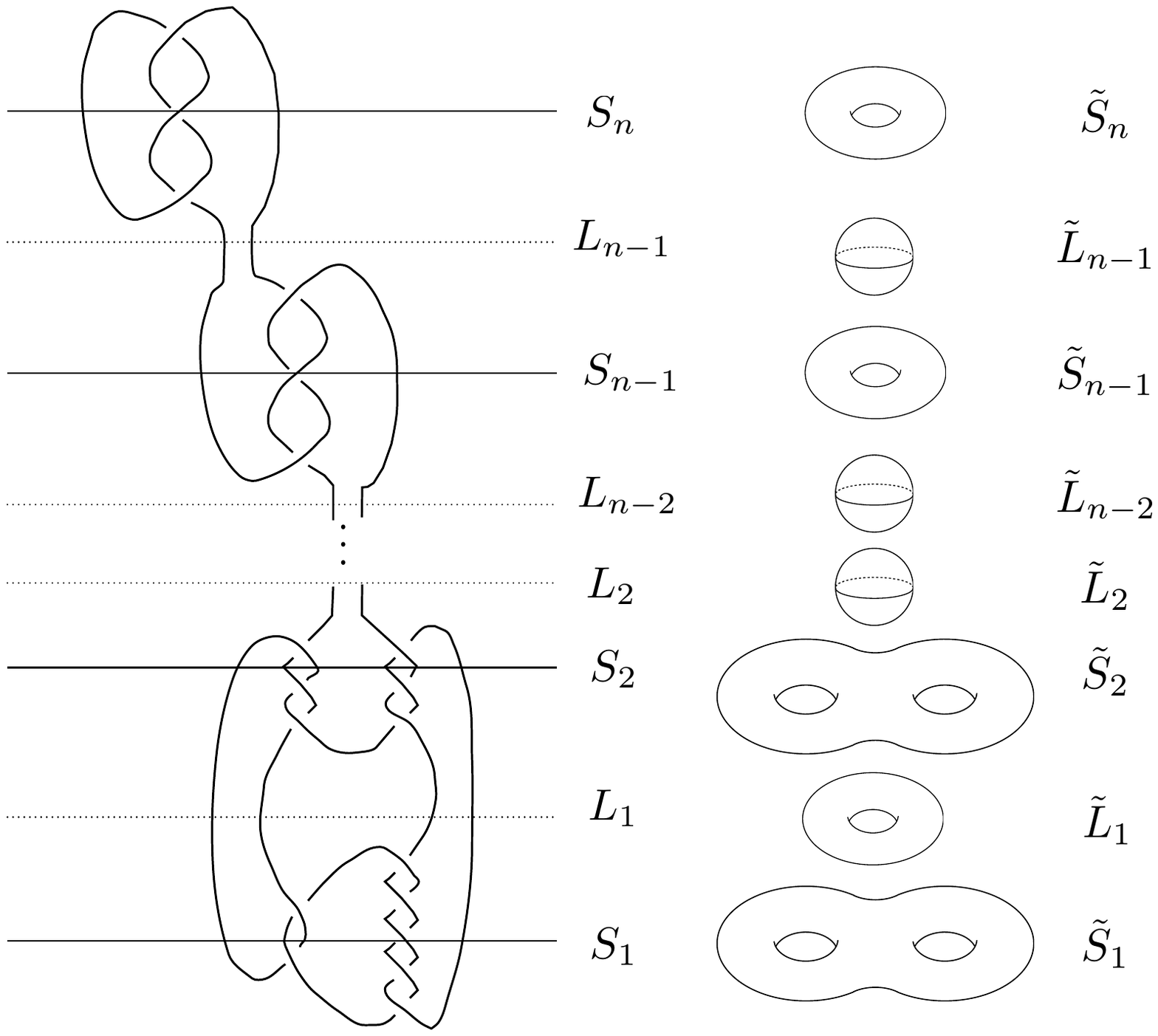}
\caption{Thin and thick levels of $k$ and corresponding surfaces in $M$\label{figure-thin}}
\end{figure}

The $3$-ball bounded by $S_1$ ($S_n$ resp.) in $S^3$ corresponds to a handlebody $H_1$ ($H_n$ resp.) in $M$, where $H_1$ ($H_n$ resp.) is bounded by $\tilde{S}_1$ ($\tilde{S}_n$ resp.) in $M$.
Moreover, the submanifold between $S_i$ and $L_{i-1}$ in $S^3$ (between $L_i$ and $S_{i}$ resp.) corresponds to a (possibly punctured) compression body $C_i^1$ ($C_i^2$ resp.) in $M$, where $\partial_+ C_i^1 = \tilde{S}_i$ ($\partial_+C_i^2=\tilde{S}_{i}$ resp.) and $\partial_- C_i^1 = \tilde{L}_{i-1}$ ($\partial_-C_i^2=\tilde{L}_{i}$ resp.) See Figure \ref{figure-md}. Moreover, $C_i^1$ ($C_i^2$ resp.) is not a trivial compression body, i.e. homeomorphic to $F\times I$ for a closed surface $F$.

\begin{figure}
\includegraphics[bb=176 83 448 550,width=4cm]{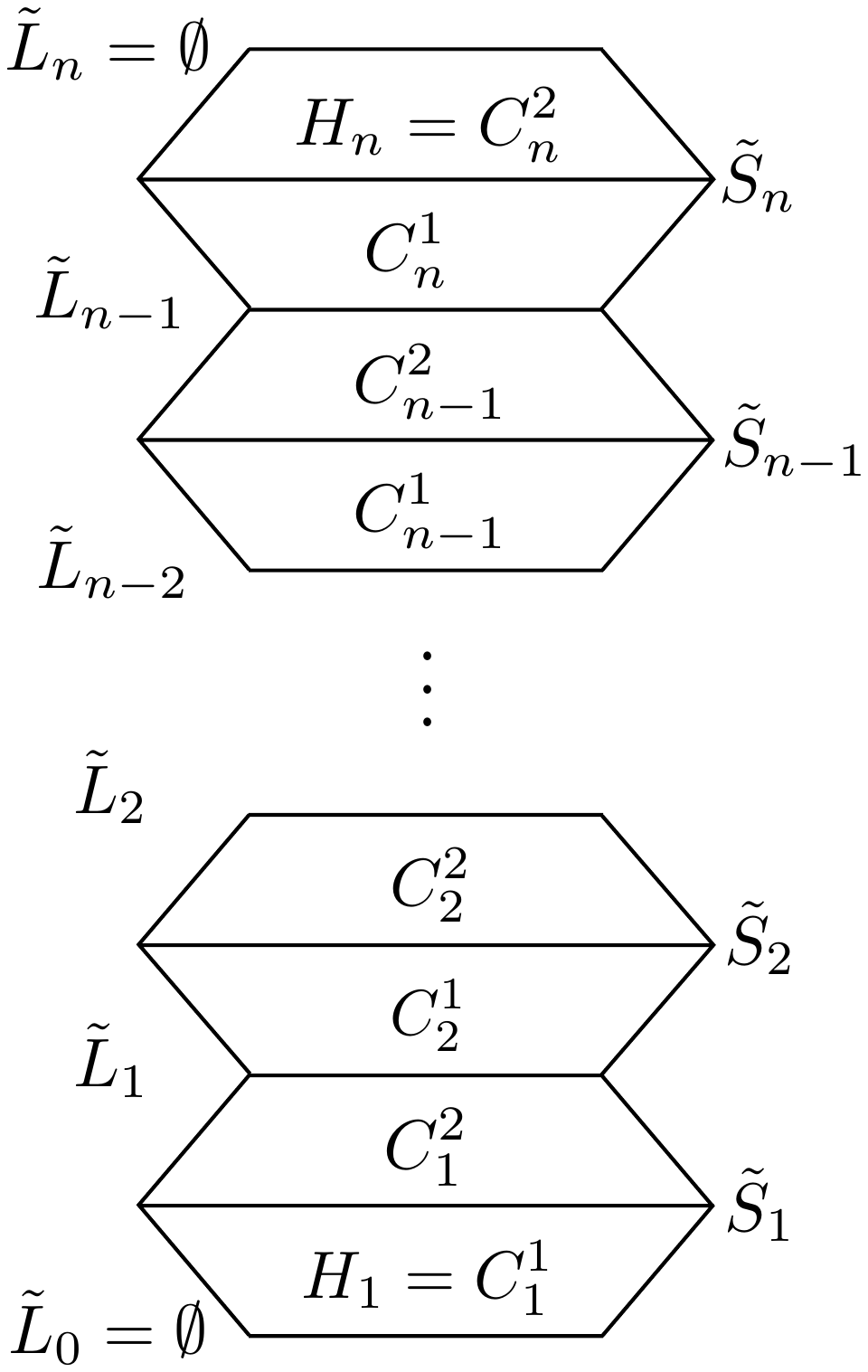}
\caption{The manifold decomposition of $M$\label{figure-md}}
\end{figure}

Therefore, the manifold decomposition of $M$ by Howards and Schultens can be written as an ordered set of thin and thick surfaces,
$$\{\tilde{L}_0=\emptyset, \tilde{S}_1, \tilde{L}_1,\tilde{S}_2,\cdots, \tilde{L}_{n-1}, \tilde{S}_n,\tilde{L}_n=\emptyset\},$$
where we denote the empty negative boundaries of $H_1$ and $H_n$ as $\tilde{L}_0$ and $\tilde{L}_n$ for convenience.

\section{Proof of Theorem \ref{theorem-1}\label{section-main-theorem}}

Let us consider the assumption that the given thin position of $k$ induces a thin manifold decomposition of DBC $M$ of $(S^3,k)$. In Rule 6 of \cite{ST}, Scharlemann and Thompson proved that each thick surface $S$ in a thin manifold decomposition is \textit{weakly incompressible}, i.e. any two compressing disks for $S$ on opposite sides of $S$ intersect along their boundary, so every thick surface in a thin manifold decomposition is strongly irreducible in $M(S)$ (we can also check it from Proposition 4.2.3 of \cite{SSS}.)

Let $H=\{\tilde{L}_0=\emptyset, \tilde{S}_1, \tilde{L}_1,\tilde{S}_2,\cdots, \tilde{L}_{n-1}, \tilde{S}_n,\tilde{L}_n=\emptyset\}$ be the manifold decomposition of DBC of $(S^3, k)$ by the H-S method as in the end of section \ref{section-DBC}, where each $\tilde{S}_i$ for $i=1,\cdots,n$ ($\tilde{L}_i$'s for $i=1,\cdots, n-1$ resp.) comes from a thick level (thin level resp.) of the given thin position of $k$.

If there exist $m$-thin levels which intersect $k$ in two points, then we get $m$-thin spheres in $H$. Since the construction of the manifold decomposition allows only one surface for each thick or thin level, these $m$-thin spheres must be $\tilde{L}_{j_1}$, $\cdots$, $\tilde{L}_{j_m}$ for some $1\leq j_1< \cdots< j_m\leq {n-1}$. Let $\tilde{L}_{j_0}$ be $\tilde{L}_0$  and $\tilde{L}_{j_{m+1}}$ be $\tilde{L}_{n}$ for convenience (so they are empty sets.)

Now we cut DBC along the $m$-thin spheres, and cap off all $S^2$ boundaries with $3$-balls. Then we get $(m+1)$-manifolds, $M_0$, $M_1$, $\cdots$, $M_m$, where each $M_i$ comes from the submanifold of DBC bounded by $\tilde{L}_{j_{i}}$ and $\tilde{L}_{j_{i+1}}$ for $i=0,\cdots, m$. In addition, we can induce the canonical manifold decomposition $H_i$ of $M_i$ from $H$, where $H_i=\{\, \emptyset, \tilde{S}_{j_i+1}, \tilde{L}_{j_i+1},\cdots,\tilde{L}_{j_{i+1}-1},\tilde{S}_{j_{i+1}},\emptyset\}$. Now every thin surface of $H_i$ is not homeomorphic to $S^2$ for $i=0,\cdots,m$, so we can say that the manifold decomposition $H_i$ of $M_i$ is a GHS in Definition of \ref{definition-GHS}. Moreover, it is obvious that the strongly irreducibility of each $\tilde{S}_j$ in $M(\tilde{S}_j)$ for $j=1,\cdots,n$ does not change after it becomes a Heegaard surface in $M_i(\tilde{S}_j)$ for some $i$. So either $H_i$ is a strongly irreducible GHS or $M_i$ is homeomorphic to $S^3$ for $i=0,\cdots,m$.

By Lemma 4.7 of \cite{Bachman3}, $M_i$ is irreducible or homeomorphic to $S^3$ (obviously, it is also irreducible) for $i=0,\cdots,m$ , i.e. no one is $S^2\times S^1$. Since the thin spheres $\tilde{L}_{j_1},\cdots,\tilde{L}_{j_m}$ are separating in $M$, no one of them is an essential sphere in a $S^2\times S^1$ piece in the prime decomposition of $M$. So we can assume that there is no $S^2\times S^1$ piece in the prime decomposition of $M$. But it is not obvious whether the sum $M_0\#_{\tilde{L}_{j_1}}M_1\#_{\tilde{L}_{j_2}}\#\cdots\#_{\tilde{L}_{j_m}}M_m$ does not have any trivial summand. So we need the following claim.\\

\Claim{No $M_i$ is homeomorphic to $S^3$ for $i=0,\cdots,m$.} 

\begin{proofc}
Suppose that $M_l$ is homeomorphic to $S^3$ for some $l$.

Let the thin levels in the thin position of $k$ corresponding to $\tilde{L}_{j_l}$ and $\tilde{L}_{j_{l+1}}$ be $L'$ and $L''$ (one of both may be empty level if $l=0$ or $m$), and the thick level in the thin position of $k$ corresponding to $\tilde{S}_{j_l+1}$ be $S$. The thin levels $L'$ and $L''$ intersect $k$ in two points, i.e. each of both $L'$ and $L''$ realizes a connected sum of $k$, i.e. $k=k_1 \#_{L'} k' \#_{L''} k_2$. 
Moreover, the thick level $S$ must intersect $k$ in $4$ or more points. 
Since $M_l\cong S^3$ is DBC of $(k',S^3)$ and $S^3$ has the unique representation as DBC of $S^3$ branched along a knot or a link (see \cite{Wa2}, and this is also true for the other lens spaces, see \cite{HR} or Problem 3.26 of \cite{K},) $k'$ is an unknot. 
Since $k$ is in a thin position, the unknot summand $k'$ in $k=k_1 \#_{L'} k' \#_{L''} k_2$ must  intersect all levels between $L'$ and $L''$ in two points, this contradicts the existence of the thick level $S$. So we get a contradiction. In the cases of $l=0$ and $m$, we get a contradiction by similar arguments for each case. This completes the proof of Claim. 
\end{proofc}

Now we can say that the thin spheres $\tilde{L}_{j_1}$, $\cdots$, $\tilde{L}_{j_m}$ determine the prime decomposition of $M$. Let us consider the thin levels $L_{j_1}$, $\cdots$, $L_{j_m}$ corresponding to $\tilde{L}_{j_1}$, $\cdots$,$\tilde{L}_{j_m}$. Then they cut $k$ into $m+1$ summands, i.e $k=k_0\#\cdots\#k_m$. In particular, each $M_i$ is DBC of $(S^3,k_i)$ for $i=0,\cdots,m$ and it is also a prime manifold as already proved. Moreover, each $k_i$ must be prime by Corollary 4 of \cite{KT}. Since $k$ is in a thin position, each summand $k_i$ must be in a thin position. This complete the proof of Theorem \ref{theorem-1}. 

\section{A generalization of Theorem \ref{theorem-1}\label{section-critical}}

In this section, we will generalize Theorem \ref{theorem-1} using the concept \textit{critical surface}. D. Bachman introduced a concept ``critical surface'' in \cite{Bachman1} and prove several properties about critical surface and minimal common stabilization. Also he proved Gordon's conjecture (see Problem 3.91 of \cite{K}) using critical surface theory in \cite{Bachman3}. In particular, he used the term \textit{topological minimal surface} to denote the class of surfaces such that they are incompressible, strongly irreducible or critical in \cite{Bachman4}.

\begin{definition}[D. Bachman, Definition 3.3 of \cite{Bachman3}]\label{definition-critical} Let $H$ be a Heegaard surface in some $3$-manifold which is compressible to both sides. The surface $H$ is \textit{critical}
if the set of all compressing disks for $H$ can be partitioned into subsets $C_0$ and $C_1$ such that the following hold.
	\begin{enumerate}
		\item For each $i=0,1$ there is at least one weak reducing pair $(V_i,W_i)$, where $V_i$, $W_i\in C_i$.
		\item If $V\in C_0$ and $W\in C_1$ then $(V,W)$ is not a weak reducing pair.
	\end{enumerate}
\end{definition}

\begin{definition}
Suppose $H$ is a GHS of a $3$-manifold $M$ with no $S^3$ components. Then $H$ is \textit{psudo-critical} if each thick surface $H_+\in \operatorname{Thick}(H)$ is strongly irreducible or critical in $M(H_+)$.\footnote{This definition is weaker than the definition ``\textit{critical GHS}'' in \cite{Bachman3}.}
\end{definition}

Now we introduce a generalization of Theorem \ref{theorem-1}.

\begin{corollary}\label{corollary-1}
If a thin position of a knot $k$ induces a manifold decomposition of double branched cover $M$ of $(S^3,k)$ by the Howards and Schultens' method where each thick surface $H_+$ of the manifold decomposition is strongly irreducible or critical in $M(H_+)$, then the thin position of $k$ is the sum of prime summands by stacking  a thin position of one of prime summands on top of a thin position of another prime summand, and so on. Therefore, $k$ holds the nearly additivity of knot width in this case.
\end{corollary}

\begin{proof}
Assume $H=\{\tilde{L}_0=\emptyset, \tilde{S}_1, \tilde{L}_1,\tilde{S}_2,\cdots, \tilde{L}_{n-1}, \tilde{S}_n,\tilde{L}_n=\emptyset\}$ be the manifold decomposition of DBC of $(S^3, k)$ by the H-S method, and $m$-thin spheres $\tilde{L}_{j_1}$, $\cdots$, $\tilde{L}_{j_m}$ cut $M$ into $M_0,\cdots,M_m$ with the canonical GHS $H_0,\cdots,H_m$ as in section \ref{section-main-theorem}. In particular, every thick surface $H_+\in \operatorname{Thick}{H}$ is strongly irreducible or critical in $M(H_+)$ by the hypothesis.

It is obvious that the strongly irreducibility or criticality of each $\tilde{S}_j$ in $M(\tilde{S}_j)$ for $j=1,\cdots,n$ does not change after it becomes a Heegaard surface in $M_i(\tilde{S}_j)$ for some $i$. So either $H_i$ is a psudo-critical GHS or $M_i$ is homeomorphic to $S^3$ for $i=0,\cdots,m$.

Since the proofs of Lemma 4.6 and Lemma 4.7 of \cite{Bachman3} do not depend on the number of critical thick levels, and only depend on the property that each thick surface of the GHS is strongly irreducible or critical, we can extend both lemmas to psudo-critical GHS. So we get each $M_i$ is irreducible or homeomorphic to $S^3$. 

The remaining arguments of the proof of Corollary \ref{corollary-1} are the same as those of Theorem \ref{theorem-1}. This completes the proof of Corollary \ref{corollary-1}
\end{proof}

\bibliographystyle{plain}

\begin{thebibliography}{10}

\bibitem{Bachman1} D. Bachman, \emph{Critical heegaard surfaces}, Trans. Amer. Math. Soc. \textbf{354} (2002), no. 10, 4015--4042 (electronic).



\bibitem{Bachman3} D. Bachman, \emph{Connected sums of unstabilized Heegaard splittings are unstabilized}, Geom. Topol. \textbf{12} (2008), no. 4,  2327--2378.

\bibitem{Bachman4} D. Bachman, \emph{Barriers to topologically minimal surfaces}, arXiv:0903.1692v1

\bibitem{BT} R. Blair and M. Tomova, \emph{Companions of the unknot and width additivity}, arXiv:0908.4103v1


\bibitem{HK} D.J. Heath and T. Kobayashi, \emph{Essential tangle decomposition from thin position of a link}, Pacific J. Math. \textbf{179} (1997), no.1, 101--117.

\bibitem{HRS} H. Howards, Y. Rieck, and  J. Schultens,
\emph{Thin position for knots and 3-manifolds: a unified approach;
Workshop on Heegaard Splittings},  Geom. Topol. Monogr. \textbf{12} (2007), 89--120. 

\bibitem{HR} C. Hodgson and and J.H. Rubinstein, \emph{Involutions and isotopies of lens spaces}, Knot theory and manifolds (Vancouver, B.C., 1983), 60--96, Lecture Notes in Math., 1144, Springer, Berlin, 1985. 

\bibitem{HS} H. Howards and J. Schultens, \emph{Thin position for knots and $3$-manifolds}, Topology Appl., \textbf{155} (2008), no. 13, 1371--1381. 

\bibitem{K} R. Kirby, \emph{Problems in low-dimensional topology}, from: “Geometric topology (Athens, GA, 1993)”, AMS/IP Stud. Adv. Math. \textbf{2}, Amer. Math. Soc. (1997) 35–473.

\bibitem{KT} P. Kim and J. Tollefson \emph{Splitting the PL involutions of nonprime 3-manifolds}, Michigan Math. J. \textbf{27} (1980) no. 3,  259--274.

\bibitem{RS} Y. Rieck and E. Sedgwick, \emph{Thin position for a connected sum of small knots}, Algebr. Geom. Topol. \textbf{2} (2002), 297--309 (electronic).

\bibitem{SSS} T. Saito, M. Scharlemann, and J. Schultens, \emph{Lecture notes on generalized Heegaard splittings}, arXiv:math/0504167.

\bibitem{ST} M. Scharlemann and A. Thompson, \emph{Thin position for $3$-manifolds}, A.M.S. Contemp. Math., 164:231-238, 1994.

\bibitem{ST2} M. Scharlemann and A. Thompson, \emph{On the additivity of knot width}, Proceedings of the Casson Fest, 135--144 (electronic), Geom. Topol. Monogr., \textbf{7}, Geom. Topol. Publ., Coventry, 2004.


\bibitem{Wa2} F. Waldhausen, \emph{\"{U}ber involutionen der 3-sph\"{a}re}, Topology \textbf{8} (1969) 81--91.

\end{thebibliography}

\end{document}